%% file: main.tex
\documentclass[letterpaper, 10 pt, conference]{ieeeconf}
\usepackage{graphicx}
\usepackage{epstopdf}
\usepackage{tikz}
\usetikzlibrary{arrows.meta,calc,positioning}

\input{mymacros}

\usepackage{xurl}
\usepackage{import}

\IEEEoverridecommandlockouts
\usepackage{amsmath, amssymb}
\usepackage{mathrsfs}

\usepackage{amsthm}

\usepackage{tikz-cd}
\tikzcdset{
  arrow style=tikz,
  diagrams={>={Stealth[length=4pt,width=3.8pt]}}
}

\theoremstyle{definition}
\newtheorem{problemEnv}{Problem}
\newenvironment{problem}[1][]
{
    \if\relax\detokenize{#1}\relax
        \begin{problemEnv}
    \else
        \begin{problemEnv}[#1]
    \fi
}
{\hfill$\diamond$\end{problemEnv}}

\theoremstyle{remark}
\newtheorem{remarkEnv}{Remark}
\newenvironment{remark}[1][]
{
    \if\relax\detokenize{#1}\relax
        \begin{remarkEnv}
    \else
        \begin{remarkEnv}[#1]
    \fi
}
{\hfill$\lozenge$\end{remarkEnv}}

\theoremstyle{plain}

\newtheorem{theoremEnv}{Theorem}

\newtheorem{lemmaEnv}[theoremEnv]{Lemma}
\newtheorem{propositionEnv}[theoremEnv]{Proposition}
\newtheorem{corollaryEnv}[theoremEnv]{Corollary}

\newenvironment{theorem}[1][]
{
    \if\relax\detokenize{#1}\relax
        \begin{theoremEnv}
    \else
        \begin{theoremEnv}[#1]
    \fi
}
{\hfill$\ $\end{theoremEnv}}

\newenvironment{lemma}[1][]
{
    \if\relax\detokenize{#1}\relax
        \begin{lemmaEnv}
    \else
        \begin{lemmaEnv}[#1]
    \fi
}
{\hfill$\ $\end{lemmaEnv}}

\newenvironment{proposition}[1][]
{
    \if\relax\detokenize{#1}\relax
        \begin{propositionEnv}
    \else
        \begin{propositionEnv}[#1]
    \fi
}
{\hfill$\ $\end{propositionEnv}}

\newenvironment{corollary}[1][]
{
    \if\relax\detokenize{#1}\relax
        \begin{corollaryEnv}
    \else
        \begin{corollaryEnv}[#1]
    \fi
}
{\hfill\end{corollaryEnv}}

\theoremstyle{definition}
\newtheorem{definitionEnv}{Definition}

\newenvironment{definition}[1][]
{
    \if\relax\detokenize{#1}\relax
        \begin{definitionEnv}
    \else
        \begin{definitionEnv}[#1]
    \fi
}
{\hfill$\bullet$\end{definitionEnv}}

\theoremstyle{myexample}
\newtheorem{exampleEnv}{Example}

\usepackage{url}
\usepackage{subcaption}
\usepackage{xcolor}

\DeclareMathOperator*{\argmin}{argmin}

\DeclareMathOperator{\Log}{Log}
\DeclareMathOperator{\Exp}{Exp}

\DeclareMathOperator{\dist}{dist}

\usepackage{algorithmic,algorithm}

\usepackage{cite}
\makeatletter
\let\NAT@parse\undefined
\makeatother
\usepackage[colorlinks=true,allcolors=blue]{hyperref}

\title{\LARGE \bf
Intrinsic Path Following on $\mathsf{SO}(3)$
}

\author{Adeel Akhtar$^{1}$\thanks{$^{1}$ Department of Mechanical and Industrial Engineering, New Jersey Institute of Technology, NJ, USA. {\tt\small adeel.akhtar@njit.edu}}}

\begin{document}

\maketitle
\thispagestyle{empty}
\pagestyle{empty}

\begin{abstract}
This paper studies intrinsic path following for fully actuated rigid-body
attitude dynamics over a broad class of attitude paths, namely arbitrary
smooth embedded closed curves on $\mathsf{SO}(3)$. We construct a
tangential--transverse state decomposition on a tubular neighborhood using
normal and pullback bundle geometry. The resulting second-order dynamics
admit a smooth feedback transformation that independently assigns
tangential and transverse covariant accelerations. A transverse
proportional--derivative law renders the path-following manifold invariant
and locally uniformly exponentially stable on an explicit forward-invariant
neighborhood, while leaving the tangential dynamics freely assignable. This
allows independent regulation of progression along the path, stabilization
of a selected attitude on the path, or tracking of a time-varying motion.
For a representative fiber path, the tubular coordinates are obtained in
closed form up to the natural cut-locus obstruction, and simulations
illustrate point stabilization and trajectory tracking.
\end{abstract}

\section{Introduction}

Rigid-body attitude control is a fundamental problem in nonlinear control,
with applications ranging from spacecraft and aerial vehicles to robotic
manipulators and autonomous systems; see, e.g.,
\cite{ChatSanMch2011,MahKumCor2012} and the references therein.
Since attitude evolves on a nonlinear configuration manifold, geometric
methods that respect this structure have become well established,
particularly for stabilization and trajectory
tracking~\cite{LeeLeoMcC2010,AkhWas2020}. While stabilization concerns
regulation to a desired attitude, trajectory tracking prescribes the
desired motion through a time-parameterized reference. In many applications,
however, the geometric route itself is important, for example when it is
chosen to satisfy pointing, visibility, or operational constraints, while
the temporal evolution along that route may remain free or be assigned
independently~\cite{AkhLaw2026CDC,LawAkh2026}.

This leads to the path-following problem, where the objective is to
stabilize a prescribed geometric path without fixing its temporal
parameterization; see~\cite{LawAkh2026} and the references therein.
Path following naturally separates two objectives: convergence transverse
to the path and progression tangential to it
\cite{NieConMagTos2008,AkhLaw2026CDC}. 
Path invariance also has a natural safety interpretation: if the prescribed
path represents an admissible motion, then an attitude initialized on the
path with velocity tangent to it remains on the path for all future time.
This viewpoint underlies transverse-feedback-linearization
approaches for nonlinear systems whose desired paths are specified in
Euclidean output spaces~\cite{NieMag04,AkhWasNie2013Journal}.
Compared with the extensive literature on attitude stabilization and
trajectory tracking, considerably less attention has been given to path
following when the path itself lies on the attitude manifold. Related
approaches include virtual holonomic constraints, guiding-vector-field
methods, and other geometric motion-control techniques
\cite{ConCosMag2017,yao2022guiding,BauGar2026,AkhLaw2026CDC}.

Our formulation is particularly close to the coordinate-free theory of
virtual holonomic constraints (VHCs) for mechanical control systems on
Riemannian manifolds~\cite{ConCosMag2017}. In that framework, an embedded
submanifold is treated as a virtual constraint and rendered invariant,
with the resulting constrained dynamics characterized through an induced
affine connection. Here, instead, we consider path following for a fully
actuated rigid body over arbitrary smooth embedded closed curves on
$\mathsf{SO}(3)$, seeking both transverse stabilization and independent
control of the motion along the path. We therefore construct an intrinsic
state decomposition on a neighborhood of the path and derive a feedback
transformation that separates the tangential and transverse accelerations.

The main contributions are as follows:
\begin{enumerate}
    \item Using tubular neighborhoods and pullback bundles, we construct an
    intrinsic tangential--transverse decomposition around a smooth embedded
    closed attitude path; see Proposition~\ref{prop:D_isomorphism}.

    \item We derive the corresponding second-order tubular dynamics and,
    using the induced tangent and normal connections, prove independent
    assignment of the tangential and transverse covariant accelerations;
    see Proposition~\ref{prop:second_order_tubular_dynamics} and
    Theorem~\ref{thm:intrinsic_tfl}.

    \item We prove path invariance and local uniform exponential stability
    of the lifted path while leaving the tangential dynamics independently
    assignable; see Proposition~\ref{prop:transverse_exponential_stability}
    and Corollary~\ref{cor:invariant_energy_sublevel}.

    \item For a fiber path on the attitude manifold, we obtain closed-form
    tubular coordinates up to the natural cut-locus obstruction; see
    Proposition~\ref{prop:fiber_tubular_coordinates}.
\end{enumerate}
\section{Geometric Preliminaries and Dynamics}
\label{sec:preliminaries}

We provide a summary of key notions used in the paper. For details, we refer readers to~\cite{BulLew2004,GalQua2020}. The configuration space of a rigid body is the special orthogonal group
$Q\eqdef\mathsf{SO}(3)=\{R\in\mathbb R^{3\times3}:R^TR=I,\ \det R=1\}$.
Its Lie algebra is
$\mathfrak{so}(3)=\{\Xi\in\mathbb R^{3\times3}:\Xi^T=-\Xi\}
=\mathsf T_I\mathsf{SO}(3)$,
where $\mathsf T_I\mathsf{SO}(3)$ denotes the tangent space of
$\mathsf{SO}(3)$ at the identity matrix $I$.
For $x\in\mathbb R^3$, let $[x]_\times\in\mathfrak{so}(3)$ denote the
skew-symmetric matrix satisfying $[x]_\times y=x\times y$ for every
$y\in\mathbb R^3$.
The tangent space of $Q$ at $R\in Q$ is
$\mathsf T_RQ=\{R[\xi]_\times:\xi\in\mathbb R^3\}$.
For a smooth manifold $M$, its tangent bundle is the vector bundle
$(\mathsf TM,M,\tau_M)$, where
$\mathsf TM\eqdef\bigsqcup_{q\in M}\mathsf T_qM$ is the disjoint union of
the tangent spaces and $\tau_M:\mathsf TM\to M$ is the canonical
projection, $\tau_M(v_q)=q$. More generally, for a smooth vector bundle
$(E,M,\pi_E)$, we denote by $\Gamma(E)$ the set of its smooth sections,
i.e., smooth maps $s:M\to E$ satisfying $\pi_E\circ s=\operatorname{id}_M$.
In particular, $\mathfrak X(M)\eqdef\Gamma(\mathsf TM)$ denotes the set
of smooth vector fields on $M$.
We use $\exp:\mathfrak{so}(3)\to\mathsf{SO}(3)$ for the matrix
exponential restricted to $\mathfrak{so}(3)$. On a Riemannian manifold
$M$, $\Exp_q:\mathcal D_q\subseteq\mathsf T_qM\to M$ denotes the
Riemannian exponential map at $q$. Whenever $\Exp_q$ is restricted to a
neighborhood $\mathcal V_q\subset\mathsf T_qM$ on which it is a
diffeomorphism onto $\mathcal U_q\subset M$, we denote its inverse by
$\Log_q:\mathcal U_q\to\mathcal V_q\subset\mathsf T_qM$.
The Riemannian norm of $v_q\in\mathsf T_qM$ is
$\|v_q\|_q\eqdef\sqrt{\langle v_q,v_q\rangle_q}$, with the base-point
subscript omitted when clear. The associated Riemannian distance is denoted by
$\dist_M:M\times M\to\mathbb R_{\ge 0}$.

The normal displacement introduced below takes values in fibers attached
to different points of the path. We therefore recall the pullback
construction~\cite{BulLew2004}.

\begin{definition}[Pullback bundle]
\label{def:pullback_bundle}
Let $(E,B,\pi_E)$ be a smooth vector bundle and $f:M\to B$ a smooth map.
The \emph{pullback bundle} of $E$ by $f$ is
\[
    f^*E
    \eqdef
    \{(x,e)\in M\times E:f(x)=\pi_E(e)\},
    \quad
    \pi_{f^*E}(x,e)=x.
\]
The canonical map $\varpi_E:f^*E\to E$ is
$\varpi_E(x,e)=e$. If $s:M\to E$ satisfies $\pi_E\circ s=f$, then
$\widetilde s(x)\eqdef(x,s(x))$ is a smooth section of $f^*E$ satisfying
$\varpi_E\circ\widetilde s=s$. These relations are summarized by
\vspace{-2pt}
\begin{equation*}
\begin{tikzcd}[
    column sep=large,
    row sep=large
]
M
    \arrow[r,"\widetilde{s}"]
    \arrow[rr,bend left=35,"s"]
    \arrow[dr,"\operatorname{id}_M"']
&
f^*E
    \arrow[r,"\varpi_E"]
    \arrow[d,"\pi_{f^*E}"']
&
E
    \arrow[d,"\pi_E"]
\\
&
M
    \arrow[r,"f"']
&
B .
\end{tikzcd}
\end{equation*}
\vspace{-2pt}
\end{definition}

We equip $Q$ with the canonical bi-invariant Riemannian metric~\cite{BulLew2004}
\begin{equation}
\label{eq:bi_invariant_metric}
    \left\langle R[\xi]_\times,R[\zeta]_\times\right\rangle_R
    \eqdef \xi^T\zeta .
\end{equation}
Let $\nabla$ denote the Levi--Civita connection associated with the
Riemannian metric~\eqref{eq:bi_invariant_metric}. Thus,
\[
\nabla:\mathfrak X(Q)\times\mathfrak X(Q)\to\mathfrak X(Q),
\qquad
(X,Y)\mapsto \nabla_XY,
\]
is the unique affine connection on $Q$ that is metric compatible and
torsion free. Since $Q$ is compact, the metric is complete and $\Exp:\mathsf TQ\to Q$ is globally defined. In particular,
$\Exp_R(R[\xi]_\times)=R\exp([\xi]_\times)$.

We consider the fully actuated rigid-body dynamics
\begin{equation}
\label{eq:rigid_body_dynamics}
    \dot R=R[\Omega]_\times,\qquad
    J\dot\Omega+\Omega\times J\Omega=\tau,
\end{equation}
where $R\in Q=\mathsf{SO}(3)$ is the attitude,
$\Omega\in\mathbb R^3$ is the body angular velocity,
$J=J^T\succ0$ is the inertia matrix, and
$\tau\in\mathbb R^3$ is the control torque~\cite{AkhWas2020,BulLew2004}.
The pair $(R,\Omega)\in Q\times\mathbb R^3$ provides a left-trivialized
representation of the tangent-bundle state $(R,V)\in\mathsf TQ$ through
$V=R[\Omega]_\times\in\mathsf T_RQ$. With $V\eqdef\dot R$, the bi-invariant metric gives $\nabla_VV=R[\dot\Omega]_\times$, and hence
\begin{equation}
\label{eq:intrinsic_mechanics}
    \nabla_VV=F(R,V)+G_R\tau,
\end{equation}
where
\[
F(R,V)=-R[J^{-1}(\Omega\times J\Omega)]_\times,
\qquad
G_R\tau=R[J^{-1}\tau]_\times.
\]
For every $R\in Q$, $G_R:\mathbb R^3\to\mathsf T_RQ$ is a linear isomorphism. The metric~\eqref{eq:bi_invariant_metric} is used to define the path geometry; it need not coincide with the kinetic-energy metric induced by $J$.

Before developing the geometric construction, we state the control
objective informally.

\noindent\textbf{Informal Problem Statement:}
Given a smooth embedded closed attitude path, we seek a smooth torque
feedback that preserves the path for tangent initial velocities, drives
nearby motions to it exponentially, and allows motion along the path to
be assigned independently of transverse stabilization. A precise
statement is given in Section~\ref{sec:intrinsic_decomposition}.

\section{Attitude Paths and Tubular Coordinates}
\label{sec:tubular_geometry}

We first introduce the geometry used to measure displacement from the desired path. Let $\gamma:\mathbb S^1\to Q$ be a smooth embedding and set
\(
    \mathcal C\eqdef\gamma(\mathbb S^1).
\)
Then $\mathcal C$ is a compact embedded one-dimensional submanifold of
$Q$, diffeomorphic to $\mathbb S^1$. Since $\mathcal C$ is diffeomorphic to $\mathbb S^1$, it is orientable.
Fix an orientation of $\mathcal C$. With respect to the induced
Riemannian metric on $\mathcal C$, let
\[
    T:\mathcal C\to\mathsf T\mathcal C
\]
denote the corresponding smooth unit tangent vector field, so that
$T(q)\in\mathsf T_q\mathcal C$ and $\|T(q)\|_q=1$ for every
$q\in\mathcal C$. For $q\in\mathcal C$, define
\begin{equation}
\label{eq:normal_space}
    N_q\mathcal C\eqdef(\mathsf T_q\mathcal C)^\perp\subset\mathsf T_qQ,
\end{equation}
and let $(N\mathcal C,\mathcal C,\pi_N)$ be the resulting rank-two normal bundle. For $\delta>0$, write
\begin{equation}
\label{eq:normal_disk_bundle}
N^{<\delta}\mathcal C
\eqdef
\bigsqcup_{q\in\mathcal C}
\{\nu_q\in N_q\mathcal C:\|\nu_q\|_q<\delta\},
\end{equation}
for the open radius-$\delta$ normal disk bundle. For each $q\in\mathcal C$, the Riemannian metric induces the orthogonal
decomposition
\[
    \mathsf T_qQ
    =
    \mathsf T_q\mathcal C
    \oplus
    N_q\mathcal C.
\]
We denote the associated orthogonal projections by
$\operatorname{pr}^{\parallel}_q:\mathsf T_qQ\to\mathsf T_q\mathcal C$
and
$\operatorname{pr}^{\perp}_q:\mathsf T_qQ\to N_q\mathcal C$.
When the base point is clear, we omit the subscript $q$ and write
$\operatorname{pr}^{\parallel}$ and $\operatorname{pr}^{\perp}$. The ambient Levi--Civita connection $\nabla$ induces a tangent connection
\[
    \nabla^{\mathcal C}:
    \mathfrak X(\mathcal C)\times\mathfrak X(\mathcal C)
    \to\mathfrak X(\mathcal C)
\]
and a normal connection
\[
    \nabla^\perp:
    \mathfrak X(\mathcal C)\times\Gamma(N\mathcal C)
    \to\Gamma(N\mathcal C).
\]
Specifically, for $X,Y\in\mathfrak X(\mathcal C)$ and
$\eta\in\Gamma(N\mathcal C)$, let $\widetilde X,\widetilde Y$, and
$\widetilde\eta$ be arbitrary smooth local extensions to $Q$. Then, for
$q\in\mathcal C$,
\begin{align}
    (\nabla_X^{\mathcal C}Y)(q)
    &\eqdef
    \operatorname{pr}^{\parallel}_q
    \big((\nabla_{\widetilde X}\widetilde Y)(q)\big),
    \\
    (\nabla_X^\perp\eta)(q)
    &\eqdef
    \operatorname{pr}^{\perp}_q
    \big((\nabla_{\widetilde X}\widetilde\eta)(q)\big).
\end{align}
These definitions are independent of the chosen extensions; moreover,
$\nabla^{\mathcal C}$ is the Levi--Civita connection of the Riemannian
metric induced on $\mathcal C$.

Restricting the Riemannian exponential to the normal bundle gives the
\emph{normal exponential map}
\[
    \Exp^\perp:N\mathcal C\to Q,
    \qquad
    \Exp^\perp(\nu_q)\eqdef\Exp_q(\nu_q).
\]
For $\delta>0$, define the corresponding tubular neighborhood by
\begin{equation}
\label{eq:tube_definition}
    \mathcal U_\delta(\mathcal C)
    \eqdef
    \Exp^\perp\!\left(N^{<\delta}\mathcal C\right).
\end{equation}
Since $\mathcal C$ is compact and embedded, the Riemannian
tubular-neighborhood theorem~\cite[Thm.~5.25, pp.~133--135]{Lee2018Riemannian}
guarantees the existence of a uniform tubular neighborhood of $\mathcal C$.

\begin{lemma}[Metric tubular neighborhood]
\label{lem:tubular_decomposition}
There exists $\delta_{\rm tub}>0$ such that
\begin{equation}
\label{eq:tubular_diffeomorphism}
    \Exp^\perp\big|_{N^{<\delta_{\rm tub}}\mathcal C}:
    N^{<\delta_{\rm tub}}\mathcal C
    \longrightarrow
    \mathcal U_{\delta_{\rm tub}}(\mathcal C)
\end{equation}
is a diffeomorphism onto an open neighborhood of $\mathcal C$. Moreover, if
$R=\Exp_q(\nu_q)$ in this neighborhood, then $q$ is the unique closest point on $\mathcal C$ and
\begin{equation}
\label{eq:metric_projection_property}
    \dist_Q(R,\mathcal C)=\dist_Q(R,q)=\|\nu_q\|_q.
\end{equation}
The same conclusions hold after restriction to $N^{<\delta}\mathcal C$ for any $0<\delta\le\delta_{\rm tub}$.
\end{lemma}

\begin{proof}
Compactness and embeddedness of $\mathcal C$ give a positive normal injectivity radius. Choosing $\delta_{\rm tub}$ below it yields the stated diffeomorphism and unique minimizing normal geodesics.
\end{proof}

Accordingly, on $\mathcal U_{\delta_{\rm tub}}(\mathcal C)$ we define
$\varepsilon:\mathcal U_{\delta_{\rm tub}}(\mathcal C)\to N^{<\delta_{\rm tub}}\mathcal C$ and
$P:\mathcal U_{\delta_{\rm tub}}(\mathcal C)\to\mathcal C$ by
\begin{equation}
\label{eq:normal_displacement_map}
\varepsilon
\eqdef
\left(\Exp^\perp\big|_{N^{<\delta_{\rm tub}}\mathcal C}\right)^{-1},
\qquad
P\eqdef\pi_N\circ\varepsilon.
\end{equation}
The map $P$ assigns to each $R$ in the tubular neighborhood its unique
closest point on the path, while $\varepsilon(R)$ is the corresponding
normal displacement from $P(R)$ to $R$.
Thus, for $R\in\mathcal U_{\delta_{\rm tub}}(\mathcal C)$,
\begin{equation}
\label{eq:tubular_reconstruction}
q=P(R),\qquad
\varepsilon(R)\in N_q\mathcal C,\qquad
R=\Exp_q(\varepsilon(R)),
\end{equation}
and $R\in\mathcal C$ if and only if $\varepsilon(R)=0_q$.

Because $\varepsilon(R)$ belongs to the normal fiber over the moving point $P(R)$, it is naturally a normal field along $P$, not a section of $N\mathcal C$ over $\mathcal C$. Applying Definition~\ref{def:pullback_bundle} with $f=P$ gives the pullback normal bundle
\[
    \big(P^*N\mathcal C,\mathcal U,\pi_{P^*N}\big),
    \qquad
    \mathcal U\eqdef\mathcal U_{\delta_{\rm tub}}(\mathcal C),
\]
whose total space is
\begin{equation}
\label{eq:pullback_normal_bundle}
    P^*N\mathcal C
    =
    \left\{
        (R,\eta)\in\mathcal U\times N\mathcal C:
        P(R)=\pi_N(\eta)
    \right\}.
\end{equation}
Its bundle projection and canonical map are
\[
    \pi_{P^*N}(R,\eta)=R,
    \qquad
    \varpi_N(R,\eta)=\eta.
\]
Since $\pi_N\circ\varepsilon=P$, the map
\begin{equation}
\label{eq:epsilon_pullback_section}
    \widetilde\varepsilon:\mathcal U\to P^*N\mathcal C,
    \qquad
    \widetilde\varepsilon(R)\eqdef(R,\varepsilon(R)),
\end{equation}
is a smooth section of the pullback bundle. The maps involved are summarized by the commutative diagram
\vspace{-2pt}
\begin{equation*}
\label{eq:pullback_normal_bundle_diagram}
\begin{tikzcd}[column sep=large,row sep=large]
&
P^*N\mathcal C
    \arrow[r,"\varpi_N"]
    \arrow[d,"\pi_{P^*N}"]
&
N\mathcal C
    \arrow[d,"\pi_N"]
\\
\mathcal U
    \arrow[ur,"\widetilde\varepsilon"]
    \arrow[r,"\operatorname{id}_{\mathcal U}"']
    \arrow[rr,bend right=20,"P"']
&
\mathcal U
    \arrow[r,"P"']
&
\mathcal C .
\end{tikzcd}
\end{equation*}
\vspace{-2pt}
In particular,
\begin{equation}
\label{eq:pullback_epsilon_identities}
    \pi_{P^*N}\circ\widetilde\varepsilon
    =\operatorname{id}_{\mathcal U},
    \qquad
    \varpi_N\circ\widetilde\varepsilon=\varepsilon,
    \qquad
    \pi_N\circ\varpi_N=P\circ\pi_{P^*N}.
\end{equation}
Thus $\widetilde\varepsilon\in\Gamma(P^*N\mathcal C)$, while
$\varepsilon$ is the corresponding normal field along $P$. This
distinction allows us to differentiate the normal displacement
intrinsically using the pullback normal connection.

\section{Intrinsic Tangential--Transverse Decomposition}
\label{sec:intrinsic_decomposition}

As $R$ varies in $\mathcal U$, a normal field along $P$ takes values in
the varying normal fibers $N_{P(R)}\mathcal C$. Thus, its rate of change
cannot be represented as an ordinary derivative in a fixed vector space
without choosing identifications between these fibers. To differentiate
such fields intrinsically, the normal connection $\nabla^\perp$ induces
the pullback connection
\[
    \bar\nabla^\perp\eqdef P^*\nabla^\perp:
    \mathfrak X(\mathcal U)\times\Gamma(P^*N\mathcal C)
    \to\Gamma(P^*N\mathcal C).
\]
Intuitively, $\bar\nabla^\perp$ differentiates sections of
$P^*N\mathcal C$ along directions in $\mathcal U$. The differential (or pushforward) of $P$ at $R$, is denoted by
$dP_R:\mathsf T_R\mathcal U\to\mathsf T_{P(R)}\mathcal C$, which maps such a
direction to the corresponding motion of the base point $P(R)$ along
$\mathcal C$.

\begin{definition}[Covariant derivative of a normal field along $P$]
\label{def:normal_field_derivative}
A \emph{smooth normal field along $P$} is a smooth map
$\eta:\mathcal U\to N\mathcal C$ satisfying
$\pi_N\circ\eta=P$. Its associated section of the pullback normal bundle is
\[
    \widetilde\eta:\mathcal U\to P^*N\mathcal C,
    \qquad
    \widetilde\eta(R)\eqdef(R,\eta(R)).
\]
For $X_R\in\mathsf T_R\mathcal U$, choose any
$X\in\mathfrak X(\mathcal U)$ satisfying $X(R)=X_R$ and define
\begin{equation}
\label{eq:pullback_normal_derivative}
    \nabla^{P,\perp}_{X_R}\eta
    \eqdef
    \varpi_N\!\left(
        (\bar\nabla_X^\perp\widetilde\eta)(R)
    \right)
    \in N_{P(R)}\mathcal C.
\end{equation}
The value of $\nabla^{P,\perp}_{X_R}\eta$ is independent of the chosen
extension $X$ because a connection is $C^\infty(\mathcal U)$-linear in
its first argument.
\end{definition}
In $\nabla^{P,\perp}_{X_R}\eta$, the superscript $P$ indicates that the derivative is taken along the map
$P:\mathcal U\to\mathcal C$, while $\perp$ indicates that the underlying
connection is the normal connection on $N\mathcal C$. 

Since $\pi_N\circ\varepsilon=P$, the normal displacement
$\varepsilon:\mathcal U\to N\mathcal C$ is a smooth normal field along
$P$. Accordingly, for each $R\in\mathcal U$, we define the linear map
\(
    A_R:\mathsf T_R\mathcal U\to N_{P(R)}\mathcal C
\)
by
\begin{equation}
\label{eq:A_definition_precise}
    A_R[X_R]\eqdef\nabla^{P,\perp}_{X_R}\varepsilon.
\end{equation}
To see its geometric meaning, consider a smooth curve $R(t)$ in
$\mathcal U$ and define
$q(t)=P(R(t))$ and $\nu(t)=\varepsilon(R(t))$. Then
\begin{equation}
\label{eq:A_curve_interpretation}
    A_{R(t)}[\dot R(t)]
    =
    \nabla^\perp_{\dot q(t)}\nu(t),
\end{equation}
where the right-hand side is the covariant derivative, with respect to
the normal connection, of the normal vector field $\nu(t)$ along
$q(t)$. Thus, $A_R[V]$ gives the intrinsic rate of change of the normal
displacement and will serve as the transverse velocity.

For a state
$(R,V)\in\mathsf T\mathcal U_{\delta_{\rm tub}}(\mathcal C)$, we define
\begin{equation}
\label{eq:intrinsic_state_variables}
    q=P(R),\quad
    \nu=\varepsilon(R),\quad
    v^\parallel=dP_R[V],\quad
    v^\perp=A_R[V].
\end{equation}
Here $q\in\mathcal C$ specifies the projected position along the path,
$\nu\in N_q\mathcal C$ is the normal displacement, and
$v^\parallel\in\mathsf T_q\mathcal C$ and
$v^\perp\in N_q\mathcal C$ are the corresponding tangential and
transverse velocities.

With these variables in place, we collect the tangential and transverse
velocity maps into the intrinsic decoupling operator
\begin{equation}
\label{eq:decoupling_operator}
\begin{aligned}
    \mathscr D_R:
    \mathsf T_R\mathcal U_{\delta_{\rm tub}}(\mathcal C)
    &\to
    \mathsf T_q\mathcal C\oplus N_q\mathcal C,
    \\
    X_R
    &\mapsto
    \big(dP_R[X_R],A_R[X_R]\big).
    \end{aligned}
\end{equation}

\begin{proposition}[Nonsingularity of the intrinsic decoupling operator]
\label{prop:D_isomorphism}
For every $R\in\mathcal U_{\delta_{\rm tub}}(\mathcal C)$, the map
$\mathscr D_R$ defined in~\eqref{eq:decoupling_operator} is a linear
isomorphism.
\end{proposition}

\begin{proof}
Fix $R\in\mathcal U$ and write
$\eta\eqdef\varepsilon(R)\in N_q\mathcal C$, where $q=P(R)$.
Let
\(
    K^\perp:\mathsf T(N\mathcal C)\to N\mathcal C
\)
denote the connector associated with the normal connection
$\nabla^\perp$. If $\eta(s)$ is a smooth curve in $N\mathcal C$ with
$\eta(0)=\eta$, $q(s)=\pi_N(\eta(s))$, and
$W=\dot\eta(0)\in\mathsf T_\eta(N\mathcal C)$, then
\begin{equation}
\label{eq:connector_defining_property}
    K^\perp(W)
    \eqdef
    \left.
    \nabla^\perp_{\dot q(s)}\eta(s)
    \right|_{s=0}
    \in N_q\mathcal C .
\end{equation}
Here $\eta(s)$ is regarded as a normal vector field along the curve
$q(s)$.
The connection determines the horizontal--vertical decomposition
\[
    \mathsf T_\eta(N\mathcal C)
    =
    H_\eta(N\mathcal C)
    \oplus
    V_\eta(N\mathcal C).
\]
Moreover, $d(\pi_N)_\eta$ restricts to an isomorphism from
$H_\eta(N\mathcal C)$ onto $\mathsf T_q\mathcal C$, while
$K^\perp$ restricts to an isomorphism from
$V_\eta(N\mathcal C)$ onto $N_q\mathcal C$. Consequently,
\begin{equation}
\label{eq:Sigma_definition}
    \Sigma_\eta:
    \mathsf T_\eta(N\mathcal C)
    \to
    \mathsf T_q\mathcal C\oplus N_q\mathcal C,
    \quad
    W\mapsto
    \big(d(\pi_N)_\eta[W],K^\perp(W)\big),
\end{equation}
is a linear isomorphism.
Since
$\varepsilon:\mathcal U\to N^{<\delta_{\rm tub}}\mathcal C$
is a diffeomorphism,
\[
    d\varepsilon_R:
    \mathsf T_R\mathcal U
    \to
    \mathsf T_\eta(N^{<\delta_{\rm tub}}\mathcal C)
    =
    \mathsf T_\eta(N\mathcal C)
\]
is a linear isomorphism. Furthermore, since
$P=\pi_N\circ\varepsilon$, the chain rule gives
\(
    dP_R
    =
    d(\pi_N)_\eta\circ d\varepsilon_R.
\)

Now let $X_R\in\mathsf T_R\mathcal U$ and choose a smooth curve
$R(s)$ in $\mathcal U$ satisfying
$R(0)=R$ and $\dot R(0)=X_R$. Setting
$\eta(s)=\varepsilon(R(s))$ gives
\(
    \dot\eta(0)=d\varepsilon_R[X_R].
\)
Using~\eqref{eq:connector_defining_property} together with
\eqref{eq:A_curve_interpretation}, we obtain
\(
    K^\perp\!\left(d\varepsilon_R[X_R]\right)
    =
    A_R[X_R].
\)
Therefore,
\begin{align}
    \Sigma_{\varepsilon(R)}
    \big(d\varepsilon_R[X_R]\big)
    &=
    \left(
        dP_R[X_R],
        A_R[X_R]
    \right) \notag\\
    &=
    \mathscr D_R[X_R].
\end{align}
Hence
\begin{equation}
\label{eq:D_as_composition}
    \mathscr D_R
    =
    \Sigma_{\varepsilon(R)}
    \circ
    d\varepsilon_R .
\end{equation}
Both factors are linear isomorphisms, and therefore
$\mathscr D_R$ is a linear isomorphism.
\end{proof}

We now lift the preceding fiberwise decomposition to the state space.
For $\eta\in N^{<\delta_{\rm tub}}\mathcal C$, the tangential and
normal velocity components are based at the point
$q=\pi_N(\eta)\in\mathcal C$. To collect these variables in a single
bundle, define
\[
    \rho
    \eqdef
    \pi_N\big|_{N^{<\delta_{\rm tub}}\mathcal C}:
    N^{<\delta_{\rm tub}}\mathcal C\to\mathcal C .
\]
The Whitney sum
$\mathsf T\mathcal C\oplus N\mathcal C$ is a vector bundle over
$\mathcal C$ whose fiber at $q\in\mathcal C$ is
$\mathsf T_q\mathcal C\oplus N_q\mathcal C$.
We then define the pullback bundle
\[
    \big(
        \mathcal Z_{\delta_{\rm tub}},
        N^{<\delta_{\rm tub}}\mathcal C,
        \pi_{\mathcal Z}
    \big)
    \eqdef
    \rho^*
    \big(
        \mathsf T\mathcal C\oplus N\mathcal C
    \big).
\]
Its total space is
\[
\begin{aligned}
    \mathcal Z_{\delta_{\rm tub}}
    =
    \bigl\{
        (\eta,w,\zeta):\;&
        \eta\in N^{<\delta_{\rm tub}}\mathcal C,\;
        w\in\mathsf T_{\pi_N(\eta)}\mathcal C,\\
        &\zeta\in N_{\pi_N(\eta)}\mathcal C
    \bigr\}.
\end{aligned}
\]
with bundle projection
$\pi_{\mathcal Z}(\eta,w,\zeta)=\eta$.

\begin{proposition}[Intrinsic state transformation]
\label{prop:intrinsic_state_diffeomorphism}
The map
\begin{equation}
\label{eq:Psi_precise}
\begin{aligned}
    \Psi:
    \mathsf T\mathcal U_{\delta_{\rm tub}}(\mathcal C)
    &\to
    \mathcal Z_{\delta_{\rm tub}},
    \\
    (R,V)
    &\mapsto
    \big(
        \varepsilon(R),
        dP_R[V],
        A_R[V]
    \big)
\end{aligned}
\end{equation}
is a diffeomorphism. Its inverse is
\begin{equation}
\label{eq:Psi_inverse_precise}
\begin{aligned}
    \Psi^{-1}:
    \mathcal Z_{\delta_{\rm tub}}
    &\to
    \mathsf T\mathcal U_{\delta_{\rm tub}}(\mathcal C),
    \\
    (\eta,w,\zeta)
    &\mapsto
    \left(
        R,\,
        \mathscr D_R^{-1}(w,\zeta)
    \right),
    \qquad
    R=\Exp^\perp(\eta).
\end{aligned}
\end{equation}
\end{proposition}
\begin{proof}
Let $(R,V)\in
\mathsf T\mathcal U_{\delta_{\rm tub}}(\mathcal C)$ and set
$\eta=\varepsilon(R)$. Since $P=\pi_N\circ\varepsilon$,
\[
    \pi_N(\eta)=P(R),
\]
and hence
$dP_R[V]\in\mathsf T_{\pi_N(\eta)}\mathcal C$ and
$A_R[V]\in N_{\pi_N(\eta)}\mathcal C$. Thus $\Psi$ is well defined.
Conversely, let
$(\eta,w,\zeta)\in\mathcal Z_{\delta_{\rm tub}}$ and set
$R=\Exp^\perp(\eta)$. Since $\varepsilon$ and
$\Exp^\perp|_{N^{<\delta_{\rm tub}}\mathcal C}$ are mutually inverse,
\[
    \varepsilon(R)=\eta,
    \qquad
    P(R)=\pi_N(\eta).
\]
Therefore
$(w,\zeta)\in
\mathsf T_{P(R)}\mathcal C\oplus N_{P(R)}\mathcal C$, and
Proposition~\ref{prop:D_isomorphism} gives the unique velocity
\(
    V=\mathscr D_R^{-1}(w,\zeta).
\)
This proves that~\eqref{eq:Psi_inverse_precise} is well defined.
Using
$\varepsilon(\Exp^\perp(\eta))=\eta$ and
$\mathscr D_R^{-1}(\mathscr D_R[V])=V$ shows directly that
$\Psi$ and~\eqref{eq:Psi_inverse_precise} are mutual inverses.
Finally, $\Psi$ is smooth by smoothness of $\varepsilon$, $P$, and $A$.
The inverse is smooth because $\Exp^\perp$ is smooth and
$R\mapsto\mathscr D_R^{-1}$ is smooth: in local bundle
trivializations, $\mathscr D_R$ is a smooth
$\mathrm{GL}(3)$-valued map, and matrix inversion is smooth on
$\mathrm{GL}(3)$. Hence $\Psi$ is a diffeomorphism.
\end{proof}

Propositions~\ref{prop:D_isomorphism} and
\ref{prop:intrinsic_state_diffeomorphism} show that the variables introduced
above form a nonsingular tangential--transverse decomposition of the
velocity and, together with the normal displacement, a diffeomorphic state
representation throughout the tubular domain. Thus, no state information
is lost in passing from $(R,V)$ to $(\nu,v^\parallel,v^\perp)$.

The lifted path-following manifold is
$\mathsf T\mathcal C\subset\mathsf TQ$. At the zero section of
$N\mathcal C$, the horizontal space of the normal connection coincides
with the tangent space to the zero section. Using
\eqref{eq:D_as_composition}, we therefore obtain the intrinsic
characterization
\begin{equation}
\label{eq:TC_characterization_precise}
    (R,V)\in\mathsf T\mathcal C
    \quad\Longleftrightarrow\quad
    \nu=0_q,\qquad v^\perp=0_q.
\end{equation}
Hence path following amounts to driving the transverse state
$(\nu,v^\perp)$ to zero while allowing $v^\parallel$ to govern motion
along the path.

\subsection{Problem Statement}
\label{sec:problem_statement}

The characterization in~\eqref{eq:TC_characterization_precise} establishes
that path following amounts to stabilizing the transverse state
$(\nu,v^\perp)$ at zero while allowing $v^\parallel$ to prescribe motion
along the path. We can therefore state the control objective as follows.

\begin{problem}[Intrinsic attitude-path following]
\label{prob:intrinsic_path_following}
Given a smooth embedded closed path
$\mathcal C\subset\mathsf{SO}(3)$ and the fully actuated rigid-body
dynamics~\eqref{eq:rigid_body_dynamics}, design a smooth feedback law
$\tau=\kappa(R,V)$, defined on a neighborhood of
$\mathsf T\mathcal C$, such that, for the resulting closed-loop system,
\begin{enumerate}
    \item[\textnormal{P1.1}]
    $\mathsf T\mathcal C$ is forward invariant;

    \item[\textnormal{P1.2}]
    the transverse zero state $(\nu,v^\perp)=(0,0)$ is locally uniformly
    exponentially stable;

    \item[\textnormal{P1.3}]
    the tangential dynamics $(q,v^\parallel)$ can be assigned independently
    of the transverse stabilizing dynamics.
\end{enumerate}
\end{problem}
\section{Intrinsic Tangential--Transverse Linearization}
\label{sec:intrinsic_tfl}

To expose how the control input acts on the intrinsic tangential and
transverse accelerations, we next differentiate $v^\parallel$ and
$v^\perp$. Since these variables belong to vector spaces attached to the
moving point $q=P(R)$, their derivatives are again defined covariantly.

We first consider the tangential component. Analogously to
Definition~\ref{def:normal_field_derivative}, the connection
$\nabla^{\mathcal C}$ induces the pullback connection
$\bar\nabla^{\mathcal C}\eqdef P^*\nabla^{\mathcal C}$ on
$P^*\mathsf T\mathcal C$, with canonical map
$\varpi_T:P^*\mathsf T\mathcal C\to\mathsf T\mathcal C$.
For a smooth tangent field
$Y:\mathcal U\to\mathsf T\mathcal C$ along $P$, let
$\widetilde Y(R)\eqdef(R,Y(R))$. For
$X_R\in\mathsf T_R\mathcal U$, define
\begin{equation}
\label{eq:pullback_tangent_derivative}
    \nabla^{P,\mathcal C}_{X_R}Y
    \eqdef
    \varpi_T\!\left(
        (\bar\nabla_X^{\mathcal C}\widetilde Y)(R)
    \right)
    \in\mathsf T_{P(R)}\mathcal C,
\end{equation}
where $X$ is any smooth extension of $X_R$. As before, the definition is
independent of the chosen extension.

For $X,Y\in\mathfrak X(\mathcal U)$, define the tangent and normal fields
along $P$ by
\[
    dP(Y)(R)\eqdef dP_R[Y(R)],
    \qquad
    A(Y)(R)\eqdef A_R[Y(R)].
\]
Their covariant derivatives are defined by
\begin{align}
\label{eq:nabla_dP_definition}
    (\nabla dP)(X,Y)
    &\eqdef
    \nabla_X^{P,\mathcal C}\big(dP(Y)\big)
    -dP(\nabla_XY),\\
\label{eq:nabla_A_definition}
    (\nabla A)(X,Y)
    &\eqdef
    \nabla_X^{P,\perp}\big(A(Y)\big)
    -A(\nabla_XY).
\end{align}
Thus, at each $R\in\mathcal U$,
\[
    (\nabla dP)_R(X_R,Y_R)\in\mathsf T_{P(R)}\mathcal C,
    \quad
    (\nabla A)_R(X_R,Y_R)\in N_{P(R)}\mathcal C.
\]
Evaluating these terms twice in the velocity direction, we define the
geometric drift
\begin{equation}
\label{eq:geometric_drift}
    \mathscr B(R,V)
    \eqdef
    \big(
        (\nabla dP)_R(V,V),
        (\nabla A)_R(V,V)
    \big)
    \in
    \mathsf T_q\mathcal C\oplus N_q\mathcal C,
\end{equation}
for $q=P(R)$. The following proposition shows how these terms enter the tangential and
transverse acceleration dynamics.

\begin{proposition}[Intrinsic second-order tubular kinematics]
\label{prop:second_order_tubular_dynamics}
Let $R(t)$ be a $C^2$ curve in $\mathcal U$, let
$V(t)\eqdef\dot R(t)$, and let
$q,\nu,v^\parallel,v^\perp$ be defined by
\eqref{eq:intrinsic_state_variables}. Then
\begin{equation}
\label{eq:first_order_tubular_dynamics}
    \dot q=v^\parallel,
    \qquad
    \nabla^\perp_{\dot q}\nu=v^\perp,
\end{equation}
and
\begin{equation}
\label{eq:second_order_tubular_dynamics}
    \begin{bmatrix}
        \nabla^{\mathcal C}_{\dot q}v^\parallel\\[1mm]
        \nabla^\perp_{\dot q}v^\perp
    \end{bmatrix}
    =
    \mathscr B(R,V)
    +
    \mathscr D_R[\nabla_VV],
\end{equation}
where $\nabla_VV\equiv\nabla_{\dot R}\dot R$ denotes the covariant
acceleration of $R(t)$.
\end{proposition}

\begin{proof}
Since $q=P(R)$, the chain rule gives
\(
    \dot q=dP_R[V]=v^\parallel.
\)
Similarly, by~\eqref{eq:A_curve_interpretation},
\(
    \nabla^\perp_{\dot q}\nu
    =
    A_R[V]
    =
    v^\perp.
\)
This proves~\eqref{eq:first_order_tubular_dynamics}.
Next, evaluating the covariant-derivative identities
\eqref{eq:nabla_dP_definition}--\eqref{eq:nabla_A_definition}
along the curve $R(t)$ in the velocity direction gives
\begin{align}
    \nabla^{\mathcal C}_{\dot q}v^\parallel
    &=
    (\nabla dP)_R(V,V)
    +
    dP_R[\nabla_VV],
    \label{eq:tangential_acceleration_identity}\\
    \nabla^\perp_{\dot q}v^\perp
    &=
    (\nabla A)_R(V,V)
    +
    A_R[\nabla_VV].
    \label{eq:normal_acceleration_identity}
\end{align}
Stacking these identities and using
\eqref{eq:geometric_drift} and
\eqref{eq:decoupling_operator} yields
\eqref{eq:second_order_tubular_dynamics}.
\end{proof}

For the rigid-body dynamics~\eqref{eq:intrinsic_mechanics}, substituting
$\nabla_VV=F(R,V)+G_R\tau$ into
Proposition~\ref{prop:second_order_tubular_dynamics} shows that the control
acts on the intrinsic covariant accelerations through the composition
$\mathscr D_R\circ G_R$. Accordingly, define the
\emph{intrinsic control-decoupling operator}
\begin{equation}
\label{eq:control_decoupling_map}
\begin{aligned}
    \mathscr A_R:
    \mathbb R^3
    &\to
    \mathsf T_q\mathcal C\oplus N_q\mathcal C,\\
    \tau
    &\mapsto
    \mathscr D_R[G_R\tau],
\end{aligned}
\qquad
\mathscr A_R
\eqdef
\mathscr D_R\circ G_R .
\end{equation}
Since $G_R:\mathbb R^3\to\mathsf T_RQ$ is a linear isomorphism by full
actuation and $\mathscr D_R$ is a linear isomorphism by
Proposition~\ref{prop:D_isomorphism}, $\mathscr A_R$ is a linear
isomorphism for every $R\in\mathcal U$. The acceleration dynamics can
therefore be written as
\begin{equation}
\label{eq:intrinsic_acceleration_control_system}
    \begin{bmatrix}
        \nabla^{\mathcal C}_{\dot q}v^\parallel\\[1mm]
        \nabla^\perp_{\dot q}v^\perp
    \end{bmatrix}
    =
    \mathscr B(R,V)
    +\mathscr D_R[F(R,V)]
    +\mathscr A_R[\tau].
\end{equation}

The geometric structure of $\mathscr A_R$ is summarized by the
commutative diagram:
\begin{equation*}
\label{eq:intrinsic_decoupling_diagram}
\begin{tikzcd}[column sep=1.7em]
\mathbb R^3
    \arrow[r,"G_R"']
    \arrow[rrr,bend left=28,"\mathscr A_R"]
&
\mathsf T_RQ
    \arrow[r,"d\varepsilon_R"]
    \arrow[rr,bend right=20,"\mathscr D_R"']
&
\mathsf T_{\varepsilon(R)}N\mathcal C
    \arrow[r,"\Sigma_{\varepsilon(R)}"]
&
\mathsf T_q\mathcal C\oplus N_q\mathcal C 
\end{tikzcd}
\end{equation*}
It follows from~\eqref{eq:D_as_composition} and
\eqref{eq:control_decoupling_map} that
\[
    \mathscr A_R
    =
    \mathscr D_R\circ G_R
    =
    \Sigma_{\varepsilon(R)}
    \circ d\varepsilon_R
    \circ G_R .
\]
Thus, the physical control torque is mapped successively to the ambient
acceleration, the variation of the tubular coordinates, and finally its
tangential and normal components.
Since $\mathscr A_R$ is a linear isomorphism throughout the tubular
domain, the tangential and transverse covariant accelerations can be
assigned independently. This gives the main intrinsic linearization
result.

\begin{theorem}[Intrinsic tangential--transverse acceleration linearization]
\label{thm:intrinsic_tfl}
Let
\(
    a^\parallel:
    \mathsf T\mathcal U
    \to
    \mathsf T\mathcal C,
    \) and
    \(
    a^\perp:
    \mathsf T\mathcal U
    \to
    N\mathcal C
\)
be smooth maps satisfying
\[
    a^\parallel(R,V)\in\mathsf T_{P(R)}\mathcal C,
    \qquad
    a^\perp(R,V)\in N_{P(R)}\mathcal C
\]
for every $(R,V)\in\mathsf T\mathcal U$. Then the feedback
\begin{equation}
\label{eq:intrinsic_feedback}
    \tau
    =
    \mathscr A_R^{-1}
    \left[
        -\mathscr B(R,V)
        -\mathscr D_R[F(R,V)]
        +
        \begin{pmatrix}
            a^\parallel(R,V)\\
            a^\perp(R,V)
        \end{pmatrix}
    \right]
\end{equation}
defines a smooth map
$\tau:\mathsf T\mathcal U\to\mathbb R^3$ and yields the intrinsic
closed-loop form
\begin{equation}
\label{eq:intrinsic_normal_form}
\begin{aligned}
    \dot q
    &=v^\parallel,
    &
    \nabla^{\mathcal C}_{\dot q}v^\parallel
    &=a^\parallel,
    \\
    \nabla^\perp_{\dot q}\nu
    &=v^\perp,
    &
    \nabla^\perp_{\dot q}v^\perp
    &=a^\perp.
\end{aligned}
\end{equation}
Hence the tangential and normal covariant accelerations can be assigned
independently throughout the tubular domain.
\end{theorem}

\begin{proof}
By~\eqref{eq:intrinsic_acceleration_control_system}, the intrinsic
acceleration dynamics satisfy
\[
    \begin{bmatrix}
        \nabla^{\mathcal C}_{\dot q}v^\parallel\\[1mm]
        \nabla^\perp_{\dot q}v^\perp
    \end{bmatrix}
    =
    \mathscr B(R,V)
    +\mathscr D_R[F(R,V)]
    +\mathscr A_R[\tau].
\]
As established above,
$\mathscr A_R:\mathbb R^3\to
\mathsf T_q\mathcal C\oplus N_q\mathcal C$
is a linear isomorphism for every $R\in\mathcal U$.
Moreover, $R\mapsto\mathscr A_R$ is smooth; hence, in local
trivializations, $\mathscr A_R$ is represented by a smooth
$\mathrm{GL}(3)$-valued matrix, and therefore
$R\mapsto\mathscr A_R^{-1}$ is smooth. Consequently,
\eqref{eq:intrinsic_feedback} defines a smooth feedback on
$\mathsf T\mathcal U$.
Substituting~\eqref{eq:intrinsic_feedback} into the acceleration
dynamics gives
\[
    \begin{bmatrix}
        \nabla^{\mathcal C}_{\dot q}v^\parallel\\[1mm]
        \nabla^\perp_{\dot q}v^\perp
    \end{bmatrix}
    =
    \begin{pmatrix}
        a^\parallel(R,V)\\
        a^\perp(R,V)
    \end{pmatrix}.
\]
Combining these identities with
\eqref{eq:first_order_tubular_dynamics} yields
\eqref{eq:intrinsic_normal_form}.
\end{proof}

For a parameterized path $\mathcal C=\gamma(\mathbb S^1)$, the
projection can be computed from
\[
    \theta^\star
    =
    \argmin_{\theta\in\mathbb S^1}
    \frac12\dist_Q^2\!\big(R,\gamma(\theta)\big),
    \qquad
    P(R)=\gamma(\theta^\star).
\]
Within the tubular domain the minimizer is unique, and, whenever
$\Log_{P(R)}$ is defined,
$\varepsilon(R)=\Log_{P(R)}R$.
The derivatives in $\mathscr D_R$ and $\mathscr B$ follow by
differentiating the closest-point optimality condition; a closed-form
special case is given in Section~\ref{sec:fiber_special_case}.
\begin{remark}[Local coordinate representation]
On a simply connected arc of $\mathcal C$, let $s$ be an arc-length
coordinate with unit tangent $T$, and choose an orthonormal normal frame
$E_1,E_2$ satisfying $\nabla_T^\perp E_i=0$. Write
\(v^\parallel=\sigma T,\;
a^\parallel=u^\parallel T,\;
\nu=z_iE_i,\;
v^\perp=\zeta_iE_i,\;
a^\perp=w_iE_i\),
with \(i\in\{1,2\}\).
Since $\nabla_T^{\mathcal C}T=0$ and
$\nabla_T^\perp E_i=0$, the intrinsic normal form
\eqref{eq:intrinsic_normal_form} becomes
\[
    \dot s=\sigma,\qquad
    \dot\sigma=u^\parallel,\qquad
    \dot z=\zeta,\qquad
    \dot\zeta=w.
\]
Thus the familiar Euclidean tangential--transverse double-integrator
form is recovered locally.
\end{remark}
\section{Transverse Stabilization and Tangential Progression}
\label{sec:stability_progression}

We first stabilize the transverse subsystem. Since
$\nu,v^\perp,a^\perp\in N_q\mathcal C$, we can define the intrinsic
PD law
\begin{equation}
\label{eq:intrinsic_PD}
    a^\perp
    =
    -k_p\nu-k_dv^\perp,
    \qquad
    k_p,k_d>0.
\end{equation}
Using
$\nabla^\perp_{\dot q}\nu=v^\perp$ and
$\nabla^\perp_{\dot q}v^\perp=a^\perp$
from~\eqref{eq:intrinsic_normal_form}, the transverse closed-loop
dynamics become
\begin{equation}
\label{eq:intrinsic_transverse_closed_loop}
    \nabla^\perp_{\dot q}
    \nabla^\perp_{\dot q}\nu
    +
    k_d\nabla^\perp_{\dot q}\nu
    +
    k_p\nu
    =
    0.
\end{equation}

\begin{proposition}[Uniform exponential transverse estimate]
\label{prop:transverse_exponential_stability}
Under the transverse feedback~\eqref{eq:intrinsic_PD}, there exist
constants $c\ge1$ and $\lambda>0$, depending only on $k_p$ and $k_d$,
such that every closed-loop solution segment contained in the tubular
domain satisfies
\begin{equation}
\label{eq:transverse_exponential_estimate}
    \|\nu(t)\|^2+\|v^\perp(t)\|^2
    \le
    ce^{-\lambda t}
    \big(
        \|\nu(0)\|^2+\|v^\perp(0)\|^2
    \big).
\end{equation}
Hence the estimate is uniform with respect to the tangential motion
$q(t)$. Moreover, the transverse zero state is invariant for as long
as the closed-loop solution remains defined.
\end{proposition}

\begin{proof}
Let $E_1(t),E_2(t)$ be the orthonormal normal frame obtained by parallel
transporting an orthonormal basis of $N_{q(0)}\mathcal C$ along the
projected curve $q(t)$. Thus,
\[
    \nabla^\perp_{\dot q}E_i=0,
    \qquad
    \langle E_i,E_j\rangle=\delta_{ij},
    \qquad i,j\in\{1,2\},
\]
where orthonormality is preserved because the normal connection is
metric compatible.
We can write
\[
    \nu(t)=\sum_{i=1}^2 z_i(t)E_i(t),
    \qquad
    z(t)\eqdef
    \begin{bmatrix}
        z_1(t)\\ z_2(t)
    \end{bmatrix}.
\]
Since the frame is parallel,
\[
    \nabla^\perp_{\dot q}\nu
    =
    \sum_{i=1}^2\dot z_iE_i,
    \qquad
    \nabla^\perp_{\dot q}
    \nabla^\perp_{\dot q}\nu
    =
    \sum_{i=1}^2\ddot z_iE_i.
\]
In particular, since
$v^\perp=\nabla^\perp_{\dot q}\nu$, the coordinates of $v^\perp$ in
this frame are $\dot z$. Therefore,
\eqref{eq:intrinsic_transverse_closed_loop} is equivalent to
\(
    \ddot z+k_d\dot z+k_pz=0.
\)
With
$x_\perp\eqdef(z,\dot z)$, this system has the constant state matrix
\(
    A_\perp
    =
    \begin{bmatrix}
        0&I_2\\
        -k_pI_2&-k_dI_2
    \end{bmatrix}.
\)
For $k_p,k_d>0$, $A_\perp$ is Hurwitz. Hence there exist
$c\ge1$ and $\lambda>0$, depending only on $k_p$ and $k_d$, such that
\[
    \|z(t)\|^2+\|\dot z(t)\|^2
    \le
    ce^{-\lambda t}
    \big(
        \|z(0)\|^2+\|\dot z(0)\|^2
    \big).
\]
Because the transported frame is orthonormal,
$\|\nu\|=\|z\|$ and $\|v^\perp\|=\|\dot z\|$. This proves
\eqref{eq:transverse_exponential_estimate}. Importantly,
$A_\perp$ is independent of the particular curve $q(t)$; hence the
constants $c$ and $\lambda$ are uniform with respect to the tangential
motion. Finally, if $(\nu(0),v^\perp(0))=(0,0)$, then
\eqref{eq:transverse_exponential_estimate} implies
$(\nu(t),v^\perp(t))=(0,0)$ for every time for which the solution is
defined.
\end{proof}

The tangential motion can be assigned independently of the transverse
feedback. Let
\(
    \kappa^\parallel:\mathsf T\mathcal C\to\mathsf T\mathcal C
\)
be a smooth fiber-preserving map, i.e.,
$\tau_{\mathcal C}\circ\kappa^\parallel=\tau_{\mathcal C}$, and choose
\[
    a^\parallel
    =
    \kappa^\parallel(q,v^\parallel)
    \in\mathsf T_q\mathcal C.
\]
For example, writing $v^\parallel=\sigma T(q)$, consider
\[
    a^\parallel
    =
    -k_\parallel(\sigma-\sigma_{\rm ref})T(q),
    \qquad
    k_\parallel>0.
\]
Since $\dot q=\sigma T(q)$ and
$\nabla_T^{\mathcal C}T=0$ on the one-dimensional Riemannian manifold
$\mathcal C$, one has
\(
    \nabla_{\dot q}^{\mathcal C}v^\parallel
    =
    \dot\sigma T(q).
\)
Hence the tangential dynamics reduce to
\(
    \dot\sigma
    =
    -k_\parallel(\sigma-\sigma_{\rm ref}),
\)
while the transverse closed-loop dynamics remain unchanged.

The exponential estimate in
Proposition~\ref{prop:transverse_exponential_stability} holds for every
solution segment that remains in the tubular domain. To obtain a
forward-time guarantee, we next identify a transverse-energy sublevel
that lies strictly inside a prescribed smaller tube. Define
\begin{equation}
\label{eq:quadratic_transverse_energy}
    \mathcal E_\perp
    \eqdef
    \frac12\|v^\perp\|^2
    +
    \frac{k_p}{2}\|\nu\|^2.
\end{equation}
Using metric compatibility of $\nabla^\perp$,
$\nabla^\perp_{\dot q}\nu=v^\perp$, and
$a^\perp=-k_p\nu-k_dv^\perp$, we obtain
\begin{align}
    \dot{\mathcal E}_\perp
    &=
    \left\langle
        v^\perp,
        \nabla^\perp_{\dot q}v^\perp
    \right\rangle
    +
    k_p
    \left\langle
        \nu,
        \nabla^\perp_{\dot q}\nu
    \right\rangle 
    =
    -k_d\|v^\perp\|^2
    \le0.
\label{eq:quadratic_transverse_energy_derivative}
\end{align}
For any $0<\delta<\delta_{\rm tub}$, define
\begin{equation}
\label{eq:invariant_energy_sublevel}
    \mathcal S_\delta
    \eqdef
    \left\{
        (R,V)\in
        \mathsf T\mathcal U_{\delta_{\rm tub}}(\mathcal C):
        \mathcal E_\perp(R,V)
        <
        \frac{k_p}{2}\delta^2
    \right\}.
\end{equation}
Since
$\mathcal E_\perp\ge(k_p/2)\|\nu\|^2$, every
$(R,V)\in\mathcal S_\delta$ satisfies $\|\nu\|<\delta$ and therefore
\(
    \mathcal S_\delta
    \subset
    \mathsf T\mathcal U_\delta(\mathcal C).
\)
Moreover, $\mathcal S_\delta$ is an open neighborhood of
$\mathsf T\mathcal C$ in
$\mathsf T\mathcal U_{\delta_{\rm tub}}(\mathcal C)$.

The decay of $\mathcal E_\perp$ guarantees that trajectories starting in
$\mathcal S_\delta$ cannot reach the boundary of the smaller tubular
neighborhood through the transverse dynamics. It remains to ensure that
the complete closed-loop trajectory exists for all forward time; this is
provided by the following result.

\begin{corollary}[Invariant tubular energy sublevel]
\label{cor:invariant_energy_sublevel}
Suppose
$a^\parallel=\kappa^\parallel(q,v^\parallel)$ and that the resulting
tangential subsystem
\[
    \dot q=v^\parallel,
    \qquad
    \nabla^{\mathcal C}_{\dot q}v^\parallel
    =
    \kappa^\parallel(q,v^\parallel)
\]
is forward complete. Then, for every
$0<\delta<\delta_{\rm tub}$, the set $\mathcal S_\delta$ is forward
invariant. Moreover, every closed-loop solution starting in
$\mathcal S_\delta$ is forward complete, remains in
$\mathsf T\mathcal U_\delta(\mathcal C)$ for all $t\ge0$, and satisfies
\eqref{eq:transverse_exponential_estimate} for all $t\ge0$.
Consequently, the transverse zero state is uniformly exponentially
stable on $\mathcal S_\delta$.
\end{corollary}
\begin{proof}
Let $[0,T_{\max})$ be the maximal forward interval of a closed-loop
solution starting in $\mathcal S_\delta$, and set
$E_0\eqdef\mathcal E_\perp(0)<(k_p/2)\delta^2$.
By~\eqref{eq:quadratic_transverse_energy_derivative},
$\mathcal E_\perp(t)\le E_0$ for all $t<T_{\max}$, and hence
\[
    \|\nu(t)\|
    \le
    \sqrt{\frac{2E_0}{k_p}}
    <\delta,
    \qquad
    \|v^\perp(t)\|
    \le
    \sqrt{2E_0}.
\]
Thus the transverse state remains bounded and strictly inside the
radius-$\delta$ tubular domain.
Suppose, by contradiction, that $T_{\max}<\infty$. Forward completeness
of the tangential subsystem implies that
$(q(t),v^\parallel(t))$ remains in a compact subset of
$\mathsf T\mathcal C$ on $[0,T_{\max}]$. Together with the bounds above
and the state diffeomorphism $\Psi$, the full trajectory therefore
remains in a compact subset of
$\mathsf T\mathcal U_\delta(\mathcal C)$.
The continuation theorem for smooth vector fields then contradicts
$T_{\max}<\infty$, so $T_{\max}=\infty$.
Finally,
$\mathcal E_\perp(t)\le\mathcal E_\perp(0)<(k_p/2)\delta^2$ for all
$t\ge0$, proving forward invariance of $\mathcal S_\delta$.
Proposition~\ref{prop:transverse_exponential_stability} then gives
\eqref{eq:transverse_exponential_estimate} for all $t\ge0$.
\end{proof}

Thus, Theorem~\ref{thm:intrinsic_tfl} and
Corollary~\ref{cor:invariant_energy_sublevel} solve
Problem~\ref{prob:intrinsic_path_following}.

\section{Closed-Form Fiber Example}
\label{sec:fiber_special_case}

We finally specialize the intrinsic construction to a fiber of the
bundle $\mathsf{SO}(3)\to\mathbb S^2$.

Let $b,d\in\mathbb S^2$, where
$\mathbb S^2\eqdef\{x\in\mathbb R^3:\|x\|=1\}$, and define
\[
    p_b:\mathsf{SO}(3)\to\mathbb S^2,
    \qquad
    p_b(R)\eqdef Rb .
\]
Choose $R_0\in\mathsf{SO}(3)$ such that $R_0b=d$, and define the
isotropy subgroup
\[
    \mathsf{SO}(2)_b
    \eqdef
    \{H\in\mathsf{SO}(3):Hb=b\}
    =
    \{\exp(\theta[b]_\times):\theta\in\mathbb R\}.
\]
The corresponding fiber is the embedded closed path
\begin{equation}
\label{eq:fiber_path}
    \mathcal C_b
    \eqdef
    p_b^{-1}(d)
    =
    R_0\mathsf{SO}(2)_b
    =
    \{R_0\exp(\theta[b]_\times):\theta\in\mathbb R\}.
\end{equation}

For $y\in\mathbb S^2\setminus\{-d\}$, let
$\mathcal R_d(y)\in\mathsf{SO}(3)$ denote the unique minimum-angle
rotation that maps $d$ to $y$. It is given by
\begin{equation}
\label{eq:min_rotation_s2}
    \mathcal R_d(y)
    =
    I+[d\times y]_\times
    +
    \frac{[d\times y]_\times^2}{1+d^Ty}.
\end{equation}
We then define
\begin{equation}
\label{eq:explicit_fiber_projection}
    P_b(R)
    \eqdef
    \mathcal R_d(Rb)^TR,
    \qquad
    Rb\neq-d.
\end{equation}
Since $\mathcal R_d(Rb)d=Rb$, one has
$P_b(R)b=d$, and hence $P_b(R)\in\mathcal C_b$.

\begin{proposition}[Explicit tubular coordinates for the fiber]
\label{prop:fiber_tubular_coordinates}
The normal exponential map restricts to a diffeomorphism
\begin{equation}
\label{eq:fiber_global_tube}
    \Exp^\perp:
    N^{<\pi}\mathcal C_b
    \longrightarrow
    \mathsf{SO}(3)\setminus p_b^{-1}(-d).
\end{equation}
For $R$ in this domain, let
$y=Rb$, $z=\Log_d^{\mathbb S^2}(y)$, and
$\xi_b(R)=P_b(R)^T(d\times z)$. Then $P_b(R)$ is the unique metric
projection of $R$ onto $\mathcal C_b$, $\xi_b(R)^Tb=0$, and
\[
    \varepsilon_b(R)
    =
    P_b(R)[\xi_b(R)]_\times
    \in N_{P_b(R)}\mathcal C_b,
    \quad
    \|\varepsilon_b(R)\|=\|z\|<\pi,
\]
with
$R=\Exp_{P_b(R)}(\varepsilon_b(R))$.
\end{proposition}

\begin{proof}
For $P\in\mathcal C_b$, we have
\[
    \mathsf T_P\mathcal C_b
    =
    \{P[\omega b]_\times:\omega\in\mathbb R\},
    \qquad
    N_P\mathcal C_b
    =
    \{P[\xi]_\times:\xi^Tb=0\}.
\]
Fix $R$ with $y=Rb\neq-d$, set $P=P_b(R)$ and
$z=\Log_d^{\mathbb S^2}(y)$, and let
$\xi=P^T(d\times z)$. Since
$\mathcal R_d(y)=\exp([d\times z]_\times)$,
one has $\xi^Tb=0$, $\|\xi\|=\|z\|<\pi$, and
\[
    \Exp_P(P[\xi]_\times)
    =
    P\exp([\xi]_\times)
    =
    \exp([d\times z]_\times)P
    =
    R.
\]
Thus every $R\notin p_b^{-1}(-d)$ has a preimage in
$N^{<\pi}\mathcal C_b$.
Moreover,
$d(p_b)_P(P[\xi]_\times)=P(\xi\times b)$, and hence
$\|d(p_b)_P(P[\xi]_\times)\|=\|\xi\|$ whenever $\xi\perp b$.
Thus $p_b$ is a Riemannian submersion on the normal spaces, giving
\[
    \dist_{\mathbb S^2}(Rb,d)
    \le
    \dist_Q(R,\widetilde P),
    \qquad
    \widetilde P\in\mathcal C_b.
\]
The normal geodesic above has length
$\|\xi\|=\|z\|=\dist_{\mathbb S^2}(Rb,d)$, so $P_b(R)$ attains this
lower bound and is therefore a nearest point.

Conversely, if
$R=P\exp([\xi]_\times)$ with
$P\in\mathcal C_b$, $\xi^Tb=0$, and $\|\xi\|<\pi$, then
$d^TRb=\cos\|\xi\|>-1$. Furthermore,
$P\exp([\xi]_\times)P^T=\exp([P\xi]_\times)$ is the minimum-angle
rotation from $d$ to $Rb$, since it has angle $\|\xi\|<\pi$ and axis
$P\xi\perp d$. Hence $P_b(R)=P$. The two constructions are smooth
mutual inverses, proving~\eqref{eq:fiber_global_tube}; injectivity also
gives uniqueness of the metric projection.
\end{proof}

Finally,
\(H(R)\eqdef R_0^TP_b(R)\in\mathsf{SO}(2)_b\)
describes progression along the fiber, while
\(\Log_d^{\mathbb S^2}(Rb)\) describes transverse displacement.
The excluded fiber \(p_b^{-1}(-d)\) is the natural cut locus, where
\(\Log_d^{\mathbb S^2}(Rb)\), and hence the normal displacement, is nonunique.
\section{Numerical Illustration}
\label{sec:simulations}

We illustrate the intrinsic construction using the fiber path of
Section~\ref{sec:fiber_special_case} with
$b=d=e_3$ and $R_0=I$, so that
\(
    \mathcal C_b
    =
    \{\exp(\theta[e_3]_\times):\theta\in\mathbb R\}.
\)
Figure~\ref{fig:point_stabilization} considers stabilization of the
attitude $R_\star=I\in\mathcal C_b$. The proposed controller first
drives the attitude toward the prescribed path and then progresses along
the path toward $R_\star$. For comparison, the geometric attitude
controller of Lee--Leok--McClamroch~\cite{LeeLeoMcC2010} is initialized from the same physical
state. While both controllers stabilize $R_\star$, the latter directly
regulates the attitude error and does not enforce attraction to the
prescribed path during the transient.
\begin{figure}[t]
    \centering
    \includegraphics[width=0.8\linewidth]{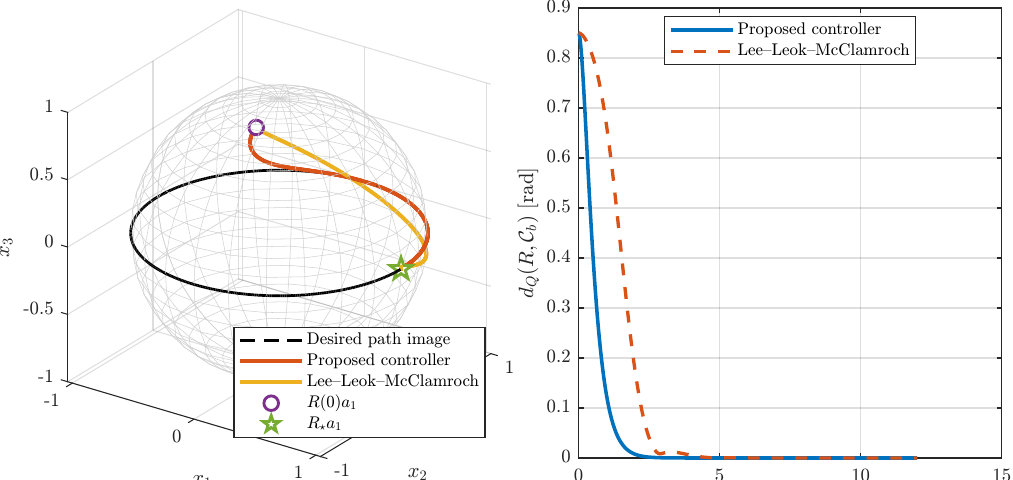}
    \vspace{-5pt}
\caption{Point stabilization at $R_\star=I$ for the proposed and
geometric attitude controllers.}
\vspace{-5pt}
    \label{fig:point_stabilization}
\end{figure}
To illustrate the decoupling established in
Theorem~\ref{thm:intrinsic_tfl}, Fig.~\ref{fig:intrinsic_decoupling}
uses the same initial transverse state and transverse gains under two
different tangential objectives: stabilization of a point on the fiber
and tracking of a moving trajectory along the fiber. Although the
tangential motions are different, the corresponding histories of
$\|\nu\|$ and $\|v^\perp\|$ coincide to numerical precision, illustrating
that the transverse closed-loop dynamics are unaffected by the
tangential assignment.

\begin{figure}[t]
    \centering
    \includegraphics[width=0.8\linewidth]{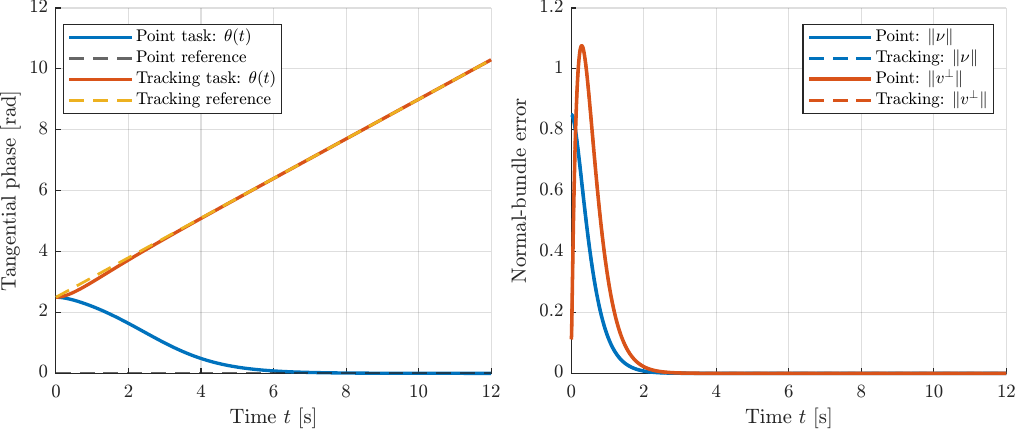}
    \vspace{-5pt}
\caption{Tangential--transverse decoupling under two different
tangential objectives.}
\vspace{-5pt}
    \label{fig:intrinsic_decoupling}
\end{figure}

\section{Conclusions}

We developed an intrinsic path-following framework for fully actuated
rigid-body attitude dynamics on $\mathsf{SO}(3)$ using tubular
neighborhoods, normal bundles, and pullback geometry. The resulting
tangential--transverse decomposition permits independent assignment of
the corresponding covariant accelerations. The proposed transverse
feedback renders the lifted path invariant and locally uniformly
exponentially stable on an explicit forward-invariant neighborhood, while
leaving the tangential dynamics freely assignable. This allows progression
along the path, stabilization of a selected attitude, or tracking of a
time-varying motion without changing the transverse closed-loop dynamics.

\textit{Disclosure of Generative AI Use: } The Large Language Model (ChatGPT) has been used to improve the syntax and grammar of some paragraphs in the manuscript.

\vspace{-2pt}
\bibliographystyle{IEEEtran}
\bibliography{my_bib}
\vspace{-2pt}
\end{document}

%% file: mymacros.tex
\newcommand{\eqdef}{\mathrel{\mathop:}=}

\newcounter{excounter}[section]

\newcounter{rmcounter}

\newcounter{dfcounter}

\newcounter{thmcounter}

\newcounter{corcounter}[thmcounter]

\newcounter{prcounter}

\newcounter{clcounter}